\documentclass[final]{ppn}

\usepackage{color,graphicx,amsmath,amsfonts,amssymb,amsxtra,setspace,xspace,mathtools,esint}
\usepackage[colorlinks=true]{hyperref}
\hypersetup{urlcolor=blue, citecolor=red, linkcolor=blue}

\newcommand{\be}[1]{\begin{equation}\label{#1}}
\newcommand{\ee}{\end{equation}}
\renewcommand{\(}{\left(}
\renewcommand{\)}{\right)}

\newcommand{\Email}[1]{E-mail: \href{mailto:#1}{\textsf{#1}}}

\newcommand{\C}{\mathbb{C}}
\newcommand{\R}{{\mathbb R}}
\renewcommand{\S}{{\mathbb S}}
\newcommand{\N}{{\mathbb N}}
\newcommand{\Z}{{\mathbb Z}}
\newcommand{\T}{{\mathbb T^2}}

\newcommand{\nA}{\nabla_{\kern-2pt\mathbf A}\,}
\newcommand{\na}{\nabla_{\kern-2pt a}\,}
\newcommand{\LapA}{\Delta_{\!\mathbf A}}
\newcommand{\HA}{\mathrm H^1_{\!\mathbf A}}

\newcommand{\nrmS}[2]{\|{#1}\|_{\mathrm L^{#2}(\S^1)}}
\newcommand{\nrmSdeux}[2]{\|{#1}\|_{\mathrm L^{#2}(\S^2)}}
\newcommand{\nrmSd}[2]{\|{#1}\|_{\mathrm L^{#2}(\S^d)}}
\newcommand{\nrmSdmun}[2]{\|{#1}\|_{\mathrm L^{#2}(\S^{d-1})}}
\newcommand{\nrmRdeux}[2]{\|{#1}\|_{\mathrm L^{#2}(\R^2)}}

\newcommand{\nrmX}[2]{\|{#1}\|_{\mathrm L^{#2}(\mathcal X)}}
\newcommand{\nrmT}[2]{\|{#1}\|_{\mathrm L^{#2}(\T)}}

\newcommand{\iS}[1]{\int_{\S^1}{#1}\,d\theta}
\newcommand{\iSsigma}[1]{\int_{\S^1}{#1}\,d\sigma}
\newcommand{\iSdeux}[1]{\int_{\S^2}{#1}\,d\sigma}
\newcommand{\iT}[1]{\iint_{\T}{#1}\,d\sigma}
\newcommand{\ird}[1]{\int_{\R^3}{#1}\,dx}
\newcommand{\irdmu}[1]{\iiint_{\R^+\times[0,2\pi)\times\R}{#1}\,d\mu}

\newcommand{\ir}[2]{\int_{\R^{#1}}{#2}\,dx}

\newcommand{\pot}{\phi}

\begin{document}
\markboth{D.~Bonheure, J.~Dolbeault, M.J.~Esteban, A.~Laptev, \& M.~Loss}
{Inequalities with Aharonov-Bohm magnetic potentials}
\title{INEQUALITIES INVOLVING AHARONOV-BOHM MAGNETIC POTENTIALS IN DIMENSIONS $2$ AND $3$}

\author{DENIS BONHEURE}
\address{D\'epartement de math\'ematique, Facult\'e des sciences, Universit\'e Libre de Bruxelles,\\
Campus Plaine CP 213, Bld du Triomphe, B-1050 Brussels, Belgium\\
\Email{denis.bonheure@ulb.be}}

\author{JEAN DOLBEAULT}
\address{CEREMADE (CNRS UMR n$^\circ$ 7534), PSL university, Universit\'e Paris-Dauphine,\\
Place de Lattre de Tassigny, 75775 Paris 16, France\\
\Email{dolbeaul@ceremade.dauphine.fr}}

\author{MARIA J.~ESTEBAN}
\address{CEREMADE (CNRS UMR n$^\circ$ 7534), PSL university, Universit\'e Paris-Dauphine,\\
Place de Lattre de Tassigny, 75775 Paris 16, France\\
\Email{esteban@ceremade.dauphine.fr}}

\author{ARI LAPTEV}
\address{Department of Mathematics, Imperial College London,\\
Huxley Building, 180 Queen's Gate, London SW7 2AZ, UK\\
\Email{a.laptev@imperial.ac.uk}}

\author{MICHAEL LOSS}
\address{School of Mathematics, Skiles Building, Georgia Institute of Technology,\\
Atlanta GA 30332-0160, USA\\
\Email{loss@math.gatech.edu}}

\maketitle
\thispagestyle{empty}
\vspace*{-8pt}

\begin{abstract} This paper is devoted to a collection of results on nonlinear interpolation inequalities associated with Schr\"odinger operators involving Aharonov-Bohm magnetic potentials, and to some consequences. As symmetry plays an important role for establishing optimality results, we shall consider various cases corresponding to a circle, a two-dimensional sphere or a two-dimensional torus, and also the Euclidean spaces of dimensions two and three. Most of the results are new and we put the emphasis on the methods, as very little is known on symmetry, rigidity and optimality in presence of a magnetic field. The most spectacular applications are new magnetic Hardy inequalities in dimensions $2$ and~$3$.\end{abstract}


\ccode{Mathematics Subject Classification 2010: Primary: 81V10, 81Q10, 35Q60.\\ Secondary: 35Q40, 49K30, 35P30, 35J10, 35Q55, 46N50, 35J20.}

\section{Introduction}\label{Sec:Introduction}

Variational problems involving magnetic fields play a peculiar role in the calculus of variations. It is fair to say that there are no simple physical intuitions that may serve as a guide and one has to extract information through exact computations, such as in the case of the Landau Hamiltonian. A case in point are the \emph{symmetry properties of optimizers} where in the absence of magnetic fields one has fairly robust methods to prove this. The isoperimetric inequality is one of the main sources of intuition; rearrangements are essentially different versions of the isoperimetric problem. This is certainly not the case in the presence of magnetic fields and there are very few results in this direction.

An example is the work of Avron, Herbst and Simon \cite{AHS} who proved that the ground state of the hydrogenic atom in a constant magnetic field has cylindrical symmetry. The proof is quite involved. Another result that comes to mind is by Erd\"os \cite{Erdos96} who proved the equivalent of the Faber-Krahn inequality for the Schr\"odinger operator with a constant magnetic field and with a Dirichlet boundary condition in a domain. The disk yields the smallest ground state energy among domains of equal area. Again, the proof is quite involved and some arguments are tailored for a linear setting. In this connection one should mention the recent result of Bonheure, Nys and van Schaftingen \cite{MR3926043} who showed perturbatively that in some nonlinear variational problem involving a small constant magnetic field the minimizers inherit the symmetry of the problem. Besides the constant magnetic field case another class of physically relevant variational problems involve Aharonov-Bohm magnetic fields and the purpose of this article is to give an up-to-date account of our knowledge.

The \emph{Aharonov-Bohm effect} states that the wave function of a charged quantum particle passing by a thin magnetic solenoid experiences a phase shift. This, despite that there is no apparent interaction with the solenoid except through the interaction of the particle with the `unphysical' vector potential. This prediction was made originally by Ehrenberg and Siday in 1949 (see~\cite{Ehrenberg}) and then again in 1959 by Aharonov-Bohm (see~\cite{Aharonov_1959}) and we will stay with the custom of calling it the Aharonov-Bohm effect. It cannot be explained in terms of classical mechanics, but was nevertheless experimentally verified (see~\cite{Batelaan_2009}). It counts as one of the important quantum mechanical effects.

Another question one may pose is the influence of the Aharonov-Bohm potential on the energies of systems, say of a particle in a potential interacting with the solenoid. It is relatively straightforward to write the Hamiltonian for this situation and one may ask for the effect of the Aharonov-Bohm field on the \emph{spectrum} of the Hamiltonian. One fruitful approach is to relate the ground state energy of a quantum mechanical particle in an external potential to the minimization of a nonlinear dual variational problem. This also works in the presence of magnetic fields and in particular with the Aharonov-Bohm field. In this context this leads to nonlinear versions of Hardy type inequalities, or simply \emph{nonlinear interpolation inequalities}, which are the dual versions of the \emph{Keller-Lieb-Thirring spectral estimates}. In various symmetry settings, we are interested in getting as much insight as possible about the best constants in the inequalities and also about the qualitative properties of their extremal functions. Indeed, in many cases studying the symmetry properties of those extremal functions allows to get very accurate, and sometimes even sharp, estimates for the best constants in the inequalities (see for instance~\cite{HS-magnetic-sharp}).

\medskip On the Euclidean space $\R^d$, the \emph{magnetic Laplacian} is defined via a \emph{magnetic potential} $\mathbf A$ by
\[
-\,\LapA\,\psi=-\Delta\,\psi-\,2\,i\,\mathbf A\cdot\nabla\psi+|\mathbf A|^2\psi-\,i\,(\mbox{div}\,\mathbf A)\,\psi\,.
\]
We consider the case of dimensions $d=2$ and $d=3$. The \emph{magnetic field} is $\mathbf B=\mbox{curl}\,\mathbf A$. The quadratic form associated with $-\,\LapA$ is given by
$\int_{\R^d}|\nA\psi|^2\,dx$ and well defined for all functions in the space
\[
\HA(\R^d):=\Big\{\psi\in\mathrm L^2(\R^d)\,:\,\nA\psi\in\mathrm L^2(\R^d)\Big\}
\]
where the \emph{magnetic gradient} takes the form
\[
\nA:=\nabla+\,i\,\mathbf A\,.
\]
The \emph{Aharonov-Bohm magnetic field} can be considered as a singular measure supported in the set $x_1=x_2=0$, where $(x_i)_{i=1}^d$ is a system of cartesian coordinates. The magnetic potential is defined as follows.
\begin{description}
\item[$\bullet$] On $\R^2$, let us consider \emph{polar coordinates} $(r,\theta)$ such that
\[
r=|x|=\sqrt{x_1^2+x_2^2}\quad\mbox{and}\quad r\,e^{i\theta}=x_1+i\,x_2
\]
and the \emph{Aharonov-Bohm magnetic potential} 
\be{AB2}
\mathbf A=\frac a{r^2}\(x_2,-\,x_1\)=\frac ar\,\mathbf e_\theta
\ee
where $a$ is a real constant and $\{\mathbf e_r,\mathbf e_\theta\}$, with $\mathbf e_r=\frac xr$,  denotes the orthogonal basis associated with our polar coordinates. The magnetic gradient and the magnetic Laplacian are explicitly given by
\[
\nA=\(\frac\partial{\partial r},\,\frac1r\(\frac\partial{\partial\theta}-\,i\,a\)\)\,,\quad-\,\LapA=-\frac{\partial^2}{\partial r^2}-\frac1r\,\frac\partial{\partial r}-\frac1{r^2}\,\(\frac\partial{\partial\theta}-i\,a\)^2\,.
\]
\item[$\bullet$] On $\R^3$, let us consider \emph{cylindrical coordinates} $(\rho,\theta,z)$ where
\[
\rho=\sqrt{x_1^2+x_2^2}\,,\quad\rho\,e^{i\theta}=x_1+i\,x_2\quad\mbox{and}\quad z=x_3
\]
and the \emph{Aharonov-Bohm magnetic potential} 
\be{AB3}
\mathbf A=\frac a{\rho^2}\(x_2,-\,x_1,0\)\,.
\ee
The magnetic gradient and the magnetic Laplacian are explicitly given by
\[
\nA=\(\frac\partial{\partial\rho},\,\frac1\rho\(\frac\partial{\partial\theta}-\,i\,a\),\,\frac\partial{\partial z}\)\,,\quad-\,\LapA=-\frac{\partial^2}{\partial\rho^2}-\frac1\rho\,\frac\partial{\partial\rho}-\frac1{\rho^2}\,\(\frac\partial{\partial\theta}-i\,a\)^2-\frac{\partial^2}{\partial z^2}\,.
\]
\end{description}
We shall also consider Aharonov-Bohm type magnetic potentials on compact manifolds, namely on the circle, the sphere and the torus. The expression of the magnetic potential will be given case by case.

\medskip This paper is intended to provide a general overview of mathematical results and methods concerning various functional inequalities involving Aharonov-Bohm magnetic fields:
\begin{itemize}
\item Magnetic ground state energy estimates.
\item Nonlinear magnetic interpolation inequalities.
\item Rigidity results for optimal functions.
\item Magnetic Keller-Lieb-Thirring inequalities.
\item Magnetic Hardy inequalities.
\end{itemize}
Except Hardy inequalities, all above inequalities will be considered on the circle~$\S^1$, on the two-dimensional sphere $\S^2$, on the two-dimensional flat torus $\T$, on $\R^2$ and on $\R^3$, with consequences on Hardy inequalities on the Euclidean spaces~$\R^2$ and~$\R^3$. It is crucial to consider precise geometric settings as we are interested in optimal inequalities, which rely on non-trivial symmetry results. A typical nonlinear interpolation inequality is
\be{NMIIsuper}
\nrmX{\nA u}2^2+\lambda\,\nrmX u2^2\ge\mu_{\mathbf A}(\lambda)\,\nrmX up^2
\ee
for any function $u$ in the appropriate $\HA$ space, for any $\lambda>0$, and for any $p>2$, say on the compact manifold $\mathcal X$ in order to fix ideas. Assuming that $\mbox{vol}(\mathcal X)=1$, the issue of \emph{the optimal inequality} is to the determine the largest value of $\lambda>0$ for which we have $\mu_{\mathbf A}(\lambda)=\lambda+C$ for some constant $C$ which is computed in terms of $|\mathbf A|$ and depends on $\mathcal X$. Equality is then realized by the constants. It is usually not difficult to prove that the equality is achieved in the inequality if $\mu_{\mathbf A}(\lambda)$ denotes the optimal constant, for any $\lambda>0$. If we consider the Euler-Lagrange equation, this can be reformulated as the slightly more general \emph{rigidity} question. For which values of $\lambda$ do we know that any solution is actually a constant? To obtain rigidity, it is essential to establish symmetry properties, which is usually the most difficult step of the proof. In the non-compact case, optimal functions are not constants, which is an additional difficulty, but the problem can also be reduced to a symmetry issue.

There is another interpretation of the rigidity issue in terms of a \emph{phase transition}. In the compact manifold case, the optimal function for~\eqref{NMIIsuper} is always a constant if $\lambda>0$ is taken small enough. For large values of $\lambda$, a simple perturbation argument shows that no minimizer can be a constant, which results in a \emph{symmetry breaking} phenomenon, and one can prove in many cases that there is a bifurcation from a symmetric phase (solutions are constant) to a non-symmetric phase for a threshold value of $\lambda$ corresponding to the optimal inequality.

Keller-Lieb-Thirring (KLT) inequalities are estimates of the ground state energy $\lambda_1[\mathbf A,V]$ for the magnetic Schr\"odinger operator $-\,\LapA-V$ in terms of $\nrmX Vq$ and are obtained by duality from~\eqref{NMIIsuper} with $q=p/(p-2)$. KLT inequalities are in fact completely equivalent to~\eqref{NMIIsuper}, including for optimality issues and related rigidity questions, and essential for proving various \emph{magnetic Hardy inequalities}, which are one of the highlights and the main motivations of this paper. However, we emphasize the fact that the accurate spectral information is carried by the KLT inequalities.

We shall actually consider not only the \emph{superquadratic} case $p>2$, but also the \emph{subquadratic} case $p<2$ in which the role of the $\mathrm L^2$ and $\mathrm L^p$ norms are exchanged. The corresponding nonlinear interpolation inequality is
\be{NMIIsub}
\nrmX{\nA u}2^2+\mu\,\nrmX up^2\ge\lambda_{\mathbf A}(\mu)\,\nrmX u2^2
\ee
for any $p\in[1,2)$ and we can also draw a whole series of consequences (rigidity, KLT, Hardy) as in the superquadratic case. In particular, we are able to prove several new and interesting results on optimality and rigidity.

\medskip Our paper collects many results on functional inequalities with magnetic fields in different geometric settings. Therefore it is difficult to pick particularly significant results. Moreover we believe that the interest of the paper lies as much in the various methods as in the results because very little is known on optimal inequalities in presence of Aharonov-Bohm magnetic fields and on the symmetry properties of the corresponding optimal functions. The most visible outcome of our work is on Hardy inequalities, which are important tools in functional analysis. The presence of a magnetic field is a key feature, for instance in dimension $d=2$. Let us anyway draw the attention of the reader to some results that are prominent in this paper:
\begin{itemize}
\item Theorem~\ref{prop:rigiditypless2} deals with nonlinear magnetic interpolation inequalities, optimal constants and rigidity results on $\S^1$ in the subquadratic case. This is a new set of inequalities which complements the theory on \emph{magnetic rings} in the superquadratic case studied in \cite{doi:10.1063/1.5022121}. It was natural to study it in view of the \emph{carr\'e du champ} technique by Bakry and Emery in~\cite{MR808640}, but as far as we know, it is an entirely new result in presence of a magnetic potential when $p<2$.
\item Theorem~\ref{prop:InterpolationTorus} is the counterpart of Theorem~\ref{prop:rigiditypless2} in the case of the torus $\T\approx\S^1\times\S^1$. It is remarkable that we achieve an optimality result here as symmetry results on products of manifolds are known to be difficult.
\item Using subquadratic interpolation inequalities for proving KLT and then Hardy type inequalities is, as far as we know, an entirely new strategy: see Theorem~\ref{Thm:ImprovedHardy2}.
\item It is natural to consider also what happens on $\S^2$: see Proposition~\ref{prop:interpolation} and Corollary~\ref{Cor:KLT22} for results in the superquadratic case.
\item Results on magnetic Hardy inequalities of Theorem~\ref{Cor:Prop:Hardy2} are a new and striking application of the nonlinear Hardy-Sobolev interpolation inequalities of~\cite{HS-magnetic-sharp} on the Euclidean space $\R^2$.
\item Theorems~\ref{Thm:Aharonov-Bohm} and~\ref{Thm:ImprovedHardy} are two examples of application of the nonlinear magnetic interpolation inequalities to magnetic Hardy inequalities on $\R^3$, which significantly improve upon the results in~\cite{MR3202866,MR4134390}.
\end{itemize}

\medskip Let us conclude this introduction by some mathematical observations and some additional references. The overall question is to determine the functional spaces which are adapted to magnetic Schr\"odinger operators in the spirit of~\cite{MR3784917}. Magnetic interpolation inequalities (without optimal constants) are usually not an issue as they can be deduced from the non-magnetic interpolation inequalities by the \emph{diamagnetic inequality}: see for instance~\cite{MR1817225}. However we are interested in retaining information about the magnetic field and characterizing optimality cases, which is by far a more difficult target. As a convention, we shall speak of \emph{Hardy-Sobolev} inequalities when a term $\int|x|^{-2}\,|u|^2\,dx$ is subtracted from the kinetic energy and of \emph{Caffarelli-Kohn-Nirenberg} inequalities when various pure power weights are taken into account. On $\R^d$, $\LapA$ has the same scaling properties as the non-magnetic Laplacian if $\mathbf A$ is a Aharonov-Bohm magnetic potential and its spectrum is explicit. KLT inequalities are spectral estimates but it has to be emphasized that they differ from semi-classical estimates as they inherit nonlinear properties of the interpolation inequalities. Finally, it is a remarkable fact that, in presence of a magnetic potential, a Hardy inequality can be established in the two-dimensional case (see \cite{doi:10.1063/1.5022121,MR4134390,HoffLa2015,MR1708811} for related papers).

\medskip This paper is organized as follows. Section~\ref{Sec:previous} is devoted to some preliminary results on the circle $\S^1$ and on the two-dimensional sphere $\S^2$. New interpolation inequalities are established on the sphere $\S^2$, with an optimality result. Section~\ref{Sec:plessthan2} is devoted to the study of a class of subquadratic magnetic interpolation inequalities on $\S^1$ and on the flat torus $\T$. We are able to identify a sharp condition of optimality and deduce several Hardy inequalities in dimensions $d=2$ and $d=3$. In Section~\ref{Sec:Interpolation $p>2$} we start by recalling earlier, non-optimal but numerically almost sharp, results on interpolation inequalities on~$\R^2$ in the presence of an Aharonov-Bohm magnetic field in order to establish some Hardy inequalities on~$\R^2$. Section~\ref{Sec:Hardy-Aharonov-Bohm} is devoted to Hardy inequalities on~$\R^3$ with singularities which are either spherically or cylindrically symmetric.

\section{General set-up and preliminary results}\label{Sec:previous}

This section is devoted to various results on the sphere $\S^d$ without magnetic field (Section~\ref{Subsec:2.1}), in the superquadratic case when $d=1$ (Section~\ref{Subsec:2.2}) and when $d=2$ (Section~\ref{SubSec:S2}). Sections~\ref{Subsec:2.1} and~\ref{Subsec:2.2} are devoted to a survey of previous results and are given not only for completeness but also as introductory material for Section~\ref{SubSec:S2} and Section~\ref{Sec:plessthan2}.

\subsection{Non-magnetic interpolation inequalities on \texorpdfstring{$\S^d$}{Sd}}\label{Subsec:2.1}

On the sphere $\S^d$, we consider the uniform probability measure $d\sigma$, which is the measure induced by the Lebesgue measure in $\R^{d+1}$, duly normalized and denote by $\nrmSd\cdot q$ the corresponding $\mathrm L^q$ norm. Here we state known results for later use.

\subsubsection{Interpolation inequalities without weights}\label{Sec:withoutWeights}
The interpolation inequalities
\be{GNsph}
\nrmSd{\nabla u}2^2\ge\frac d{p-2}\(\nrmSd up^2-\nrmSd u2^2\)
\ee
hold for any $p\in[1,2)\cup(2,+\infty)$ if $d=1$ and $d=2$, and for any $p\in[1,2)\cup(2,2^*]$ if $d\ge3$, where $2^*:=2\,d/(d-2)$ is the Sobolev critical exponent. See~\cite{MR1230930,MR1134481} for $p>2$ and~\cite{MR808640} if $d=1$ or $d\ge2$ and $p\le\big(2\,d^2+1\big)/(d-1)^2$.

If $p>2$, we know from~\cite{MR3218815} that there exists a concave monotone increasing function $\lambda\mapsto \mu_{0,p}(\lambda)$ on $(0,+\infty)$ such that $\mu_{0,p}(\lambda)$ is the optimal constant in the inequality
\be{SuperLinear}
\nrmSd{\nabla u}2^2+\lambda\,\nrmSd u2^2\ge\mu_{0,p}(\lambda)\,\nrmSd up^2\quad\forall\,u\in\mathrm H^1(\S^d)
\ee
and that $\mu_{0,p}(\lambda)=\lambda$ if and only if $\lambda\le d/(p-2)$. In this range, equality is achieved if and only if $u$ is a constant function: this is a \emph{symmetry} range. On the opposite, if $\lambda>d/(p-2)$, the optimal function is not constant and we shall say that there is \emph{symmetry breaking}.

The case $1\le p<2$ is similar: there exists a concave monotone increasing function $\mu\mapsto\lambda_{0,p}(\mu)$ on $(0,+\infty)$ such that $\lambda_{0,p}(\mu)$ is the optimal constant in the inequality
\be{Sublinear}
\nrmSd{\nabla u}2^2+\mu\,\nrmSd up^2\ge\lambda_{0,p}(\mu)\,\nrmSd u2^2\quad\forall\,u\in\mathrm H^1(\S^d)
\ee
and that $\lambda_{0,p}(\mu)=\mu$ if and only if $\mu\le d/(2-p)$. In this symmetry range, constants are the optimal functions, while there is symmetry breaking if $\mu>d/(2-p)$: optimal functions are non-constant.

In the symmetry range, positive constants are actually the only \emph{positive} solutions of the Euler-Lagrange equation
\[
-\,\varepsilon\,\Delta u+\lambda\,u=u^{p-1}
\]
where $\varepsilon=\pm1$ is the sign of $(p-2)$, while there are multiple solutions in the symmetry breaking range. The limit case $p=2$ can be obtained by taking the limit as $p\to2$ and the corresponding inequality is the logarithmic Sobolev inequality. Much more is known and we refer to~\cite{MR3218815} for further details.

\subsubsection{A weighted Poincar\'e inequality for the ultra-spherical operator}
Using cylindrical coordinates $(z,\omega)\in[-1,1]\times\S^{d-1}$, we can rewrite the Laplace-Beltrami operator on $\S^d$ as
\[
\Delta=\mathcal L_d+\frac1{1-z^2}\,\Delta_\omega\quad\mbox{with}\quad\mathcal L_d\,u:=\(1-z^2\)u''-d\,z\,u'
\]
where $\Delta_\omega$ denotes the Laplace-Beltrami operator on $\S^{d-1}$ and $\mathcal L_d$ is the \emph{ultra-spherical operator}. In other words, $\mathcal L_d$ is the Laplace-Beltrami operator on $\S^d$ restricted to functions which depend only on $z$. The operator $\mathcal L_d$ has a basis of eigenfunctions $G_{\ell,d}$, the Gegenbauer polynomials, associated with the eigenvalues $\ell\,(\ell+d-1)$ for any $\ell\in\N$ (see~\cite{MR0372517}). Here $d$ is not necessarily an integer.

Let us consider the eigenvalue problem
\be{Ev0}
-\,\mathcal L_2\,f+\frac{4\,\mathsf a^2}{1-z^2}\,f=\lambda\,f\,.
\ee
By changing the unknown function according to $f(z)=\(1-z^2\)^{\mathsf a}\,g(z)$, we obtain that $g$ solves
\[
-\,\mathcal L_{2\,(2\,\mathsf a+1)}\,g+2\,\mathsf a\,(1+2\,\mathsf a)\,g=\lambda\,g
\]
which determines the eigenvalues $\lambda=\lambda_{\ell,\mathsf a}$ given by
\be{LambdaEllA}
\lambda_{\ell,\mathsf a}=\ell\,\big(\ell+2\,(2\,\mathsf a+1)-1\big)+2\,\mathsf a\,(1+2\,\mathsf a)=(\ell+2\,\mathsf a)\,(\ell+2\,\mathsf a+1)\,,\quad\ell\in\N\,.
\ee
We shall denote by $g_{\ell,\mathsf a}(z)=G_{\ell,2\,(2\,\mathsf a+1)}(z)$ the associated eigenfunctions and define $f_{\ell,\mathsf a}(z):=\(1-z^2\)^{\mathsf a}g_{\ell,\mathsf a}(z)$. By considering the lowest positive eigenvalue, we obtain a \emph{weighted Poincar\'e inequality}.
\begin{lemma}\label{Lemma:LdWPoincare} For any $\mathsf a\in\R$ and any function $f\in\mathrm H_0^1[-1,1]$, we have
\[
\int_{-1}^1\(\(1-z^2\)\big|f'(z)\big|^2+\frac{4\,\mathsf a^2}{1-z^2}\,|f(z)|^2\)dz\ge\lambda_{1,\mathsf a}\int_{-1}^1\,\big|f(z)-\bar f(z)\big|^2\,dz\,,
\]
where
\[
\bar f(z)=\(1-z^2\)^{\mathsf a}\,\frac{\int_{-1}^1f(z)\,\(1-z^2\)^{\mathsf a}\,dz}{\int_{-1}^1\(1-z^2\)^{2\,\mathsf a}\,dz}\,.
\]
Equality is achieved by a function $f$ if and only if $f$ is proportional to $f_{1,\mathsf a}(z)=z\,\(1-z^2\)^{\mathsf a}$.\end{lemma}
Notice for consistency that, if $f(z)=\(1-z^2\)^{\mathsf a}\,g(z)$, then
\begin{multline*}
\int_{-1}^1\(\(1-z^2\)\big|f'(z)\big|^2+\frac{4\,\mathsf a^2}{1-z^2}\,|f(z)|^2\)dz\\
=\int_{-1}^1\(\(1-z^2\)\big|g'(z)\big|^2+2\,\mathsf a\,(1+2\,\mathsf a)\,|g(z)|^2\)\(1-z^2\)^{2\,\mathsf a}\,dz\,,
\end{multline*}
where the right-hand side is the Dirichlet form associated with the operator
\[
-\,\mathcal L_{2\,(2\,\mathsf a+1)}+2\,\mathsf a\,(1+2\,\mathsf a)\,.
\]

\subsection{Magnetic rings: superquadratic inequalities on \texorpdfstring{$\S^1$}{S1}}\label{Subsec:2.2}

In this section, we review a series of results which have been obtained in~\cite{doi:10.1063/1.5022121} in the superquadratic case $p>2$, in preparation for an extension to the subquadratic case $p\in[1,2)$ that will be studied in Section~\ref{Sec:plessthan2}.

\subsubsection{Magnetic interpolation inequalities and consequences}\label{SubSec:MR}
Let us consider the superquadratic case $p>2$ in dimension $d=1$. We recall that $d\sigma=(2\pi)^{-1}\,d\theta$ where $\theta\in[0,2\pi)\approx\mathbb S^1$. As in~\cite{doi:10.1063/1.5022121} we consider the space $\mathrm H^1(\S^1)$ of the $2\pi$-periodic functions $u\in C^{0,1/2}(\S^1)$, such that $u'\in\mathrm L^2(\S^1)$. Inequality~\eqref{SuperLinear} can be rewritten as
\be{Ineq:InterpZero0}
\nrmS{u'}2^2+\lambda\,\nrmS u2^2\ge\lambda\,\nrmS u p^2\quad\forall\,u\in\mathrm H^1(\S^1)
\ee
for any $\lambda\in(0,1/(p-2)]$. We also have the inequality
\be{Ineq:InterpZero2}
\nrmS{u'}2^2+\frac14\,\nrmS{u^{-1}}2^{-2}\ge\frac14\,\nrmS u2^2\quad\forall\,u\in\mathrm H^1(\S^1)\,,
\ee
according to~\cite{MR1708787}, with the convention that $\nrmS{u^{-1}}2^{-2}=0$ if $u^{-2}$ is not integrable and, as a special case, if $u$ changes sign. Notice that inequality~\eqref{Ineq:InterpZero2} is formally the case $p=2\,d/(d-2)$ and $\lambda=d/(p-2)$ of~\eqref{SuperLinear} when $d=1$ (see~\cite[Appendix~A]{doi:10.1063/1.5022121}).

In~\cite{doi:10.1063/1.5022121}, it was shown that the inequality (for complex valued functions)
\be{Ineq:Interp0}
\nrmS{\psi'-i\,a\,\psi}2^2+\lambda\,\nrmS\psi2^2\ge\mu_{a,p}(\lambda)\,\nrmS\psi p^2\quad\forall\,\psi\in\mathrm H^1(\S^1,\mathbb C)
\ee
is equivalent, after eliminating the phase, to the inequality
\[
\nrmS{u'}2^2+a^2\,\nrmS{u^{-1}}2^{-2}+\lambda\,\nrmS u2^2\ge\mu_{a,p}(\lambda)\,\nrmS up^2\quad\forall\,u\in\mathrm H^1(\S^1)\,.
\]
The equivalence is relatively easy to prove if $\psi$ does not vanish, but some care is required otherwise: see~\cite{doi:10.1063/1.5022121} for details. Here we denote by $\mu_{a,p}(\lambda)$ the optimal constant in~\eqref{Ineq:Interp0}. Using~\eqref{Ineq:Interp0} and then~\eqref{Ineq:InterpZero0}, we obtain that
\begin{multline*}
\nrmS{u'}2^2+a^2\,\nrmS{u^{-1}}2^{-2}+\lambda\,\nrmS u2^2\\
=(1-4\,a^2)\,\nrmS{u'}2^2+\lambda\,\nrmS u2^2+4\,a^2\,\big(\nrmS{u'}2^2+\tfrac14\,\nrmS{{u}^{-1}}2^2\big)\\
\ge(1-4\,a^2)\,\big(\nrmS{u'}2^2+\tfrac{a^2+\lambda}{1-4\,a^2}\,\nrmS u2^2\big)\ge(a^2+\lambda)\,\nrmS up^2
\end{multline*}
under the condition $(a^2+\lambda)/(1-4\,a^2)\le1/(p-2)$, which provides an estimate of $\mu_{a,p}(\lambda)$. This estimate turns out to be optimal.
\begin{proposition}[\cite{doi:10.1063/1.5022121}]\label{prop:rigidity} Let $p>2$, $a\in[0,1/2]$, and $\lambda>-\,a^2$.
\begin{enumerate}
\item[(i)] If $a^2\,(p+2)+\lambda\,(p-2)\le1$, then $\mu_{a,p}(\lambda)=a^2+\lambda$ and equality in~\eqref{Ineq:Interp0} is achieved only by the constants.
\item[(ii)] If $a^2\,(p+2)+\lambda\,(p-2)>1$, then $\mu_{a,p}(\lambda)<a^2+\lambda$ and equality in~\eqref{Ineq:Interp0} is not achieved by the constants.
\end{enumerate}
\end{proposition}
The condition $a\in[0,1/2]$ is not a restriction. First, replacing $\psi$ by $e^{iks}\,\psi(s)$ for any $k\in\mathbb Z$ shows that $\mu_{a+k,p}(\mu)=\mu_{a,p}(\mu)$ so that we can assume that $a\in[0,1]$. Then by considering $\chi(s)=e^{-is}\,\overline\psi(s)$, we find that
\[
|\psi'+i\,a\,\psi|^2=|\chi'+i\,(1-a)\,\chi|^2\,,
\]
hence $\mu_{a,p}(\mu)=\mu_{1-a,p}(\mu)$.

\subsubsection{Magnetic Hardy inequalities on~\texorpdfstring{$\S^1$}{S1} and~\texorpdfstring{$\R^2$}{R2}}
As in~\cite{doi:10.1063/1.5022121}, we can draw an easy consequence of Proposition~\ref{prop:rigidity} on a Keller-Lieb-Thirring type inequality. By H\"older's inequality applied with $q=p/(p-2)$, we have
\[
\nrmS{\psi'-i\,a\,\psi}2^2-\mu^{-1}\,\iS{\pot\,|\psi|^2}\ge\nrmS{\psi'-i\,a\,\psi}2^2-\mu^{-1}\,\nrmS\pot q\,\nrmS\psi p^2\,.
\]
Using~\eqref{Ineq:Interp0} with $\lambda=0$ and $\mu$ such that $\mu^{-1}\,\nrmS\pot q=\mu_{a,p}(0)$, we know that the right-hand side is nonnegative. See~\cite{doi:10.1063/1.5022121} for more details. Altogether we obtain the following \emph{magnetic Hardy inequality on $\S^1$: for any $a\in\R$, any $p>2$ and $q=p/(p-2)$, if~$\pot$ is a non-trivial potential in $\mathrm L^q(\S^1)$, then}
\be{Ineq:KLT1}
\nrmS{\psi'-i\,a\,\psi}2^2\ge\frac{\mu_{a,p}(0)}{\nrmS\pot q}\iSsigma{\pot\,|\psi|^2}\quad\forall\,\psi\in\HA(\S^1)\,.
\ee
This is a special case of the more general interpolation inequality
\begin{multline}\label{Ineq:KLT1spectral}
\nrmS{\psi'-i\,a\,\psi}2^2-\iS{\pot\,|\psi|^2}\ge\nrmS{\psi'-i\,a\,\psi}2^2-\mu\,\nrmS\psi p^2\\
\ge-\,\lambda_{a,p}(\mu)\,\nrmS\psi2^2
\end{multline}
with $\mu=\nrmS\pot q$, where we denote by $\lambda_{a,p}(\mu)$ the inverse function of $\lambda\mapsto\mu_{a,p}(\lambda)$, as defined in Proposition~\ref{prop:rigidity}. See~\cite{MR3784917} for details.

\medskip The standard non-magnetic Hardy inequality on $\R^d$, \emph{i.e.},
\[
\int_{\R^d}|\nabla\psi|^2\,dx\ge\frac14\,(d-2)^2\int_{\R^d}\frac{|\psi|^2}{|x|^2}\,dx\quad\forall\,\psi\in\mathrm H^1(\R^d)\,,
\]
degenerates if $d=2$, but this degeneracy is lifted in the presence of a Aharonov-Bohm magnetic field. According to~\cite{MR1708811}, we have
\[\label{HardyAB}
\int_{\R^2}|\nA\psi|^2\,dx\ge\min_{k\in\mathbb Z}\,(a-k)^2\int_{\R^2}\frac{|\psi|^2}{|x|^2}\,dx\quad\forall\,\psi\in\mathrm H^1(\R^d)\,.
\]
It is natural to ask whether an improvement can be obtained if the singularity $|x|^{-2}$ is replaced by a weight which has an angular dependence. Using polar coordinates $x\approx(r,\theta)$ and interpolation inequalities of~\cite{MR3218815}, the inequality
\[
\int_{\R^d}|\nabla\psi|^2\,dx\ge\frac{(d-2)^2}{4\,\nrmSdmun\varphi q}\int_{\R^d}\frac{\varphi(\theta)}{|x|^2}\,|\psi|^2\,dx\quad\forall\,\psi\in\mathrm H^1(\R^d)
\]
was proved in~\cite{HoffLa2015}, under the condition that $q\ge1+\frac12\,(d-2)^2/(d-1)$, again with normalized measure on $\S^{d-1}$. Magnetic and non-radial improvements have been combined in~\cite{doi:10.1063/1.5022121}. Let us give a statement in preparation for similar extensions to the case of dimension $d=3$.
\begin{corollary}[\cite{doi:10.1063/1.5022121}]\label{Cor:HardyMR} Let $\mathbf A$ as in~\eqref{AB2}, $a\in\left[0,1/2\right]$, $p>2$, $q=p/(p-2)$ and assume that $\varphi$ is a non-negative function in $\mathrm L^q(\S^1)$. With the above notations, the inequality
\[\label{Hardy1}
\int_{\R^2}|\nA\psi|^2\,dx\ge\tau\int_{\R^2}\frac{\varphi(\theta)}{|x|^2}\,|\psi|^2\,dx\quad\forall\,\psi\in\mathrm H^1(\R^d)
\]
holds with a constant $\tau>0$ which is the unique solution of the equation
\[
\lambda_{a,p}(\tau\,\nrmS\varphi q)=0\,.
\]
Moreover, $\tau=a^2/\nrmS\varphi q$ if $a^2\le1/(p+2)$.\end{corollary}

\subsection{Magnetic interpolation inequalities on~\texorpdfstring{$\S^2$}{S2}}\label{SubSec:S2}

In the spirit of Sections~\ref{Subsec:2.1} and~\ref{Subsec:2.2}, we state some new results concerning the two-dimensional sphere with main results in Proposition~\ref{prop:interpolation} and Corollary~\ref{Cor:KLT22}.

\subsubsection{A magnetic ground state estimate}
Let us consider the magnetic Laplacian on $\S^2$ and the associated Dirichlet form $\int_{\S^2}|\nA u|^2\,d\sigma$ where $d\sigma$ is the uniform probability measure on $\S^2$. Using cylindrical coordinates $(\theta, z)\in[0,2\pi)\times [-1,1]$, we can write that $d\sigma=\frac1{4\,\pi}\,dz\,d\theta$ and assume that the magnetic gradient takes the form
\be{nablaS2}
\nA u=\(\begin{array}{c}\sqrt{1-z^2}\,\frac{\partial u}{\partial z}\\
\frac1{\sqrt{1-z^2}}\(\frac{\partial u}{\partial\theta}-i\,a\,u\)
\end{array}\)
\ee
where $a>0$ is a magnetic flux, so that
\[
|\nA u|^2=\(1-z^2\)\left|\frac{\partial u}{\partial z}\right|^2+\frac1{1-z^2}\,\left|\frac{\partial u}{\partial\theta}-i\,a\,u\,\right|^2\,.
\]
\begin{lemma}\label{lem:gapp} Assume that $a\in\R$. With the notation~\eqref{nablaS2}, we have
\[
\iSdeux{|\nA u|^2}\ge\Lambda_a\iSdeux{|u|^2}\quad\forall\,u\in\HA(\S^2)
\]
with optimal constant
\be{Lambdaa}
\Lambda_a=\min_{k\in\Z}|k-a|\,\big(|k-a|+1\big)\,.
\ee\end{lemma}
Notice that $\Lambda_a\le\Lambda_{1/2}=3/4$.
\begin{proof} We can write $u$ using a Fourier decomposition
\[
u(z,\theta)=\sum_{\ell\in\N}\sum_{k\in\Z}u_{k,\ell}(z)\,e^{i\,k\,\theta}
\]
and observe that
\[
|\nA u|^2=\sum_{\ell\in\N}\sum_{k\in\Z}\(\(1-z^2\)\left|u_{k,\ell}'(z)\right|^2+\frac{(k-a)^2}{1-z^2}\,\left|u_{k,\ell}(z)\right|^2\)
\]
where
\[
u_{k,\ell}(z)=f_{\ell,|k-a|/2}(z)\iSdeux{u(z,\theta)\,\frac{2\,f_{\ell,|k-a|/2}(z)\,e^{-i\,k\,\theta}}{\int_{-1}^1\(1-z^2\)^{|k-a|}\,dz}}
\]
and $f_{\ell,|k-a|/2}$ is an eigenfunction of~\eqref{Ev0} with eigenvalue $\lambda=\lambda_{\ell,\mathsf a}$ such that $2\,\mathsf a=|k-a|$. Using~\eqref{LambdaEllA}, we conclude that the spectrum of $-\,\LapA$ is given by
\[
\big(\ell+|k-a|\big)\,\big(\ell+|k-a|+1\big)\,,\quad k\in\Z\,,\;\ell\in\N\,.
\]
\end{proof}

\subsubsection{Superquadratic interpolation inequalities and consequences}
\begin{proposition}\label{prop:interpolation} Let $a\in\R$ and $p>2$. With the notation~\eqref{nablaS2}, there exists a concave monotone increasing function $\lambda\mapsto \mu_{a,p}(\lambda)$ on $(-\Lambda_a,+\infty)$ such that $\mu_{a,p}(\lambda)$ is the optimal constant in the inequality
\[
\nrmSdeux{\nA u}2^2+\lambda\,\nrmSdeux u2^2\ge\mu_{a,p}(\lambda)\,\nrmSdeux up^2\quad\forall\,u\in\HA(\S^2).
\]
Furthermore, $\mu_{a,p}(\lambda)\ge2\,(\lambda+\Lambda_a)/\big(2+(p-2)\,\Lambda_a\big)$ and $\lim_{\lambda\to-\Lambda_a}\mu_{a,p}(\lambda)=0$, with~$\Lambda_a$ given by~\eqref{Lambdaa}.
\end{proposition}
\begin{proof} The proof is adapted from~\cite[Proposition~3.1]{MR3784917}. For an arbitrary $t\in(0,1)$, we can write that
\begin{align*}
\nrmSdeux{\nA u}2^2+\lambda\,\nrmSdeux u2^2\ge&\,t\(\nrmSdeux{\nA u}2^2-\Lambda_a\,\nrmSdeux u2^2\)\\
&+(1-t)\(\nrmSdeux{\nabla|u|}2^2+\frac{\lambda+t\,\Lambda_a}{1-t}\,\nrmSdeux u2^2\)\\
&\kern12pt\ge(1-t)\,\mu_{0,p}\(\frac{\lambda+t\,\Lambda_a}{1-t}\)\,\nrmSdeux up^2\,,
\end{align*}
as a consequence of Lemma~\ref{lem:gapp} and of the \emph{diamagnetic inequality} (see e.g.~\cite[Theorem~7.21]{MR1817225})
\[
\nrmSdeux{\nA u}2^2\ge\nrmSdeux{\nabla|u|}2^2\,.
\]
If $\lambda<2/(p-2)$, the estimate is obtained by choosing $t$ such that
\[
\frac{\lambda+t\,\Lambda_a}{1-t}=\frac2{p-2}
\]
and recalling that $\mu_{0,p}\(2/(p-2)\)=2/(p-2)$. The limit as $\lambda\to-\Lambda_a$ is obtained by taking the ground state of $-\LapA$ on $\mathrm H^1(\S^2)$ as test function. \end{proof}

With the same method as for the proof of~\eqref{Ineq:KLT1}, we can deduce a Hardy-type inequality.
\begin{corollary}\label{Cor:KLT22} Let $a\in\R$, $p>2$ and $q=p/(p-2)$. With the notation~\eqref{nablaS2}, if~$\pot$ is a non-trivial potential in $\mathrm L^q(\S^2)$, then
\[
\nrmSdeux{\nA u}2^2\ge\frac{\mu_{a,p}(0)}{\nrmSdeux \pot q}\iSdeux{\pot\,|u|^2}\quad\forall\,u\in\HA(\S^2)\,.
\]
\end{corollary}

\section{Subquadratic magnetic interpolation inequalities}\label{Sec:plessthan2}

This section is devoted to new and optimal interpolation inequalities involving $\mathrm L^p$ norms in the \emph{subquadratic} range $p\in(1,2)$ on the circle and on the torus (Sections~\ref{Sec:magnetic rings $p<2$} and~\ref{Sec:flat}). Dual Keller-Lieb-Thirring inequalities yield new magnetic Hardy inequalities on $\R^2$ and $\R^3$ (Section~\ref{Sec:Consequences}).

\subsection{Magnetic rings: subquadratic interpolation inequalities on \texorpdfstring{$\S^1$}{S1}}\label{Sec:magnetic rings $p<2$}

We extend the results of Section~\ref{SubSec:MR} to the subquadratic range $1<p<2$ using the strategy of~\cite{doi:10.1063/1.5022121}.

\subsubsection{Statement of the inequality}
As a special case of~\eqref{GNsph} corresponding to $d=1$, we have the non-magnetic interpolation inequality
\be{Ineq:InterpZero}
(2-p)\,\nrmS{u'}2^2+\nrmS u p^2\ge\nrmS u2^2\quad\forall\,u\in\mathrm H^1(\S^1)
\ee
for any $p\in[1,2)$. Our first result is the magnetic counterpart of this inequality.
\begin{lemma}\label{Lem:SubMagInterp} Let $a\in\R$ and $p\in[1,2)$. Then there exists a concave monotone increasing function $\mu\mapsto\lambda_{a,p}(\mu)$ on $\R^+$ such that
\be{Ineq:Interp}
\nrmS{\psi'-i\,a\,\psi}2^2+\mu\,\nrmS\psi p^2\ge\lambda_{a,p}(\mu)\,\nrmS\psi2^2\quad\forall\,\psi\in\mathrm H^1(\S^1,\mathbb C)\,.
\ee\end{lemma}
Here we denote by $\lambda_{a,p}(\mu)$ the optimal constant in~\eqref{Ineq:InterpZero}.
\begin{proof} The existence of $\lambda_{a,p}(\mu)$ is a consequence of~\eqref{Ineq:InterpZero} and of the \emph{diamagnetic inequality}: let $\rho=|\psi|$ and $\phi$ be such that $\psi=\rho\big(\theta)\,\exp(i\,\phi(\theta)\big)$. Since
\[
|\psi'-i\,a\,\psi|^2=|\rho'|^2+|\phi'-a|^2\,\rho^2\ge|\rho'|^2\,,
\]
we have that $\nrmS{\psi'-i\,a\,\psi}2^2\ge\nrmS{\,|\psi|'\,}2^2$. The concavity of $\mu\mapsto\lambda_{a,p}(\mu)$ is a consequence of the definition of $\lambda_{a,p}(\mu)$ as the optimal constant, \emph{i.e.}, the infimum on $\mathrm H^1(\S^1)\ni\psi$ of an affine function of $\mu$.\end{proof}

\subsubsection{Existence of an optimal function}
\begin{lemma}\label{Lem:Existence} For all $a\in [0, 1/2]$, $p\in[1,2)$ and $\mu\ge-\,a^2$, equality in~\eqref{Ineq:Interp} is achieved by at least one function in $\mathrm H^1(\S^1)$.\end{lemma}
\begin{proof} We consider a minimizing sequence $\{\psi_n\}$ for
\[
\lambda_{a,p}(\mu)=\inf\Big\{\nrmS{\psi'-i\,a\,\psi}2^2+\mu\,\nrmS\psi p^2\,:\,\psi\in\mathrm H^1(\S^1)\,,\;\nrmS\psi2=1\Big\}\,.
\]
By the diamagnetic inequality we know that the sequence $(\psi_n)_{n\in\N}$ is bounded in $\mathrm H^1(\S^1)$. By the compact Sobolev embeddings, this sequence is relatively compact in $\mathrm L^p(\S^1)$ and in $\mathrm L^2(\S^1)$. The map $\psi\mapsto\nrmS{\psi'-i\,a\,\psi}2^2$ is lower semicontinuous by Fatou's lemma, which proves the claim.\end{proof}

\subsubsection{A non-vanishing property}
\begin{lemma}\label{Lem:NonVanishing} Asssume that $a\in (0, 1/2)$, $p\in[1,2)$ and $\mu\ge-\,a^2$. If $\psi\in\mathrm H^1(\S^1)$ is an optimal function for~\eqref{Ineq:Interp} with $\nrmS\psi p=1$, then $\psi(s)\neq0$ for any $s\in\S^1$.\end{lemma}
\begin{proof} The proof goes as in~\cite{doi:10.1063/1.5022121}. Let us decompose $v(s)=\psi(s)\,e^{ias}$ as a real and an imaginary part, respectively $v_1$ and $v_2$, which both solve the same Euler-Lagrange equation
\[
-\,v_j''-\,\mu\,\big(v_1^2+v_2^2\big)^{\frac p2-1}\,v_j=\lambda_{a,p}(\mu)\,v_j\,,\quad j=1\,,\;2\,.
\]
Notice that $v\in C^{0,1/2}(\S^1)$ and the nonlinear term is continuous, hence $v$ is smooth. The Wronskian $w=(v_1\,v_2'-\,v_1'\,v_2)$ is constant. If both $v_1$ and $v_2$ vanish at the same point, then $w$ vanishes identically, which means that $v_1$ and $v_2$ are proportional. With $a\in (0, 1/2)$, $\psi$ is not $2\pi$-periodic, a contradiction.\end{proof}

\subsubsection{A reduction to a scalar minimization problem}
We refer to Section~\ref{Sec:withoutWeights} if $a=0$ and assume in the proofs that $a>0$. The main steps of the reduction are similar to the case $p>2$ of~\cite{doi:10.1063/1.5022121}. We repeat the key points for completeness. Let us define
\[
\mathcal Q_{a,p,\mu}[u]:=\frac{\nrmS{u'}2^2+a^2\,\nrmS{u^{-1}}2^{-2}+\mu\,\nrmS up^2}{\nrmS u2^2}\,.
\]
Notice that if $u\in\mathrm H^1(\S^1)$ is such that $u(s_0)=0$ for some $s_0\in(-\pi,\pi]$, then
\[
|u(s)|^2=\(\int_{s_0}^su'\,ds\)^2\le\sqrt{2\pi}\,\nrmS{u'}2\,\sqrt{|s-s_0|}
\]
and $u^{-2}$ is not integrable. In this case, as mentioned earlier, we adopt the convention that
\be{extendedexpression}
\mathcal Q_{a,p,\mu}[u]:=\frac{\nrmS{u'}2^2+\mu\,\nrmS up^2}{\nrmS u2^2}\,.
\ee
\begin{lemma}\label{lem:reduction} For any $a\in[0,1/2]$, $p\in[1,2)$, $\mu>-\,a^2$,
\[
\lambda_{a,p}(\mu)=\min_{u\in\mathrm H^1(\S^1)\setminus\{0\}}\mathcal Q_{a,p,\mu}[u]\,.
\]
\end{lemma}
\begin{proof} We consider functions on $\S^1$ as $2\pi$-periodic functions on $\R$. If $\psi\in\mathrm H^1(\S^1)$, then $v(s)=\psi(s)\,e^{ias}$ satisfies the condition
\be{Cdt:Per}
v(s+2\pi)=e^{2i\pi a}\,v(s)\quad\forall\,s\in\R
\ee
and
\[
\lambda_{a,p}(\mu)=\min\frac{\nrmS{v'}2^2+\mu\,\nrmS vp^2}{\nrmS v2^2}
\]
where the minimization is taken on the set of the functions $v\in C^{0,1/2}(\R)$ such that $v'\in\mathrm L^2(-\pi,\pi)$ and~\eqref{Cdt:Per} holds.

With $v=u\,e^{i\phi}$ written in polar form, the boundary condition becomes
\be{bcuphi}
u(\pi)=u(-\pi)\,,\quad\phi(\pi)=2\pi\,(a+k)+\phi(-\pi)
\ee
for some $k\in\mathbb Z$, and $\nrmS{v'}2^2=\nrmS{u'}2^2+\nrmS{u\,\phi'}2^2$ so that
\[
\lambda_{a,p}(\mu)=\min\frac{\nrmS{u'}2^2+\nrmS{u\,\phi'}2^2+\mu\,\nrmS up^2}{\nrmS u2^2}
\]
where the minimization is taken on the set of the functions $(u,\phi)\in C(\R)^2$ such that $u'$, $u\,\phi'\in\mathrm L^2(\S^1)$ and~\eqref{bcuphi} holds.

Up to a multiplication of $u$ by a constant so that $\nrmS up=1$, the Euler-Lagrange equations are
\[
-\,u''+|\phi'|^2\,u+\mu\,|u|^{p-2}\,u=\lambda_{a,p}(\mu)\,u\quad\mbox{and}\quad(\phi'\,u^2)'=0\,.
\]
If $a\in(0,1/2)$, by integrating the second equation and using Lemma~\ref{Lem:NonVanishing}, we find a constant $L$ such that $\phi'=L/u^2$. Taking~\eqref{bcuphi} into account, we deduce from
\[
L\int_{-\pi}^\pi\frac{ds}{u^2}=\int_{-\pi}^\pi\phi'\,ds=2\pi\,(a+k)
\]
that
\[
\nrmS{u\,\phi'}2^2=L^2\int_{-\pi}^\pi\frac{ds}{u^2}=\frac{(a+k)^2}{\nrmS{u^{-1}}2^2}\,.
\]
This establishes that
\[
\lambda_{a,p}(\mu)=\min_{u,\,k}\mathcal Q_{a+k,p,\mu}[u]
\]
where the minimization is taken on all $k\in\Z$ and on all functions $u\in\mathrm H^1(\S^1)$. Because of the restriction $a\in(0,1/2)$, the minimum is achieved by $k=0$.

The case $a=1/2$ is a limit case that can be handled as in~\cite[Theorem~III.7]{doi:10.1063/1.5022121}. In this case the result holds also true, with the minimizer being in $\mathrm H^1_0(\S^1)\setminus\{0\}$, and with the convention defined in \eqref{extendedexpression} for the expression of $Q_{a,p,\mu}[u]$ when $u$ vanishes in $\S^1$.\end{proof}

\subsubsection{A rigidity result}
If $a\in(0,1/2)$, as in~\cite{doi:10.1063/1.5022121}, the study of~\eqref{Ineq:Interp} is reduced to the study of the inequality
\be{Ineq:Interp2}
\nrmS{u'}2^2+a^2\,\nrmS{u^{-1}}2^{-2}+\mu\,\nrmS up^2\ge\lambda_{a,p}(\mu)\,\nrmS u2^2\quad\forall\,u\in\mathrm H^1(\S^1)
\ee
where $u$ is now a real valued function. Necessary adaptations to the trivial case $a=0$ and to the limit case $a=1/2$ are straightforward and left to the reader. The result below is the analogue of Proposition~\ref{prop:rigidity} in the subcritical range.
\begin{theorem}\label{prop:rigiditypless2} Let $p\in (1,2)$, $a\in(0,1/2)$, and $\mu>0$.
\begin{enumerate}
\item[(i)] If $\mu\,(2-p)+4\,a^2\le1$, then $\lambda_{a,p}(\mu)=a^2+\mu$ and equality in~\eqref{Ineq:Interp2} is achieved only by the constants.
\item[(ii)] If $\mu\,(2-p)+4\,a^2>1$, then $\lambda_{a,p}(\mu)<a^2+\mu$ and equality in~\eqref{Ineq:Interp2} is not achieved by the constants.
\end{enumerate}
\end{theorem}
\begin{proof} In case (i) we can write
\begin{multline*}
\nrmS{u'}2^2+a^2\,\nrmS{u^{-1}}2^{-2}+\mu\,\nrmS up^2\\
=(1-4\,a^2)\(\nrmS{u'}2^2+\frac\mu{1-4\,a^2}\,\nrmS up^2\)\\
+4\,a^2\(\nrmS{u'}2^2+\frac14\,\nrmS{u^{-1}}2^2\)\,,
\end{multline*}
deduce from~\eqref{Ineq:InterpZero} that
\[
\nrmS{u'}2^2+\frac\mu{1-4\,a^2}\,\nrmS up^2\ge\frac\mu{1-4\,a^2}\,\nrmS u2^2
\]
if $\mu/(1-4\,a^2)\le1/(2-p)$ and conclude using~\eqref{Ineq:InterpZero2}.

In case (ii), let us consider the test function $u_\varepsilon:=1+\varepsilon\,w_1$, where $w_1$ is the eigenfunction corresponding to the first non-zero eigenvalue of $-\,d^2/ds^2$ on $\mathrm H^1(\S^1)$, with periodic boundary conditions, namely, \hbox{$w_1(s)=\cos s$} and $\lambda_1=1$. A Taylor expansion shows that
\[
\mathcal Q_{a,p,\mu}[u_\varepsilon]=\big(1+a^2-\,\mu\,(2-p)\big)\,\varepsilon^2+o(\varepsilon^2)\,,
\]
which proves the result. Notice that the Taylor expansion is also valid if $a=0$, so that $(p-2)$ is the optimal constant in~\eqref{Ineq:InterpZero}, and also that a similar Taylor expansion holds in case of~\eqref{Ineq:InterpZero2}, which formally corresponds to $p=-\,2$.
\end{proof}

\subsection{Aharonov-Bohm magnetic interpolation inequalities on \texorpdfstring{$\T$}{T2}}\label{Sec:flat}

We consider a toy model for Aharonov-Bohm magnetic fields on the flat torus which can be seen as a $2$-dimensional cylinder, with periodicity along the axis. It is a $2$-dimensional extension of the superquadratic magnetic ring model of Section~\ref{Sec:magnetic rings $p<2$}.

Let us consider the flat torus $\T=\S^1\times\S^1\approx[-\pi,\pi)\times[-\pi,\pi)\ni(x,y)$ with periodic boundary conditions in $x$ and $y$. We denote by $d\sigma$ the uniform probability measure $d\sigma=dx\,dy/(4\pi^2)$ and consider the magnetic gradient
\be{nablaT2}
\nA\psi:=\big(\psi_x,\psi_y-i\,a\,\psi\big)
\ee
and the magnetic kinetic energy
\[
\nrmT{\nA u}2^2=\iT{|\nA\psi|^2}=\iT{\(|\psi_x|^2+|\psi_y-\,i\,a\,\psi|^2\)}\,.
\]

\subsubsection{A magnetic ground state estimate}
\begin{lemma}\label{lem:GroundStateTorus} Assume that $a\in[0,1/2]$. With the notation~\eqref{nablaT2}, we have
\[
\iT{|\nA\psi|^2}\ge a^2\iT{|\psi|^2}\quad\forall\,\psi\in\HA(\T)\,.
\]\end{lemma}
\begin{proof} We make a Fourier decomposition on the basis $(e^{i\,\ell\,x}\,e^{i\,k\,y})_{k,\ell\in\Z}$. We find that the lowest modes are given by
\begin{align*}
&k=0\,,\;\ell=0\,:\;\lambda_{00}=a^2\,,\\
&k=1\,,\;\ell=0\,:\;\lambda_{10}=(1-a)^2\ge a^2\, \text{ since } a\in[0,1/2],\\
&k=0\,,\;\ell=1\,:\;\lambda_{01}=1+a^2\,.
\end{align*}
Therefore, $\lambda_{00}$ is the lowest mode.
\end{proof}

\subsubsection{The Bakry-Emery method applied to the 2-dimensional torus}\label{Subsubsection:3.2.2}
We consider the flow given by
\[
\frac{\partial u}{\partial t}=\Delta u+(p-1)\,\frac{|\nabla u|^2}u
\]
and observe that
\[
\frac d{dt}\,\nrmT{u(t,\cdot)}p^2=0
\]
on the one hand, and
\begin{multline*}
-\,\frac12\,\frac d{dt}\(\nrmT{\nabla u(t,\cdot)}2^2-\lambda\,\nrmT{u(t,\cdot)}2^2\)\\
=\nrmT{\Delta u}2^2+(p-1)\iT{\Delta u\,\frac{|\nabla u|^2}u}-\lambda\,(2-p)\,\nrmT{\nabla u}2^2
\end{multline*}
on the other hand. Integrations by parts show that
\[
\nrmT{\Delta u}2^2=\nrmT{\mathrm{Hess}\,u}2^2
\]
and
\[
\iT{\Delta u\,\frac{|\nabla u|^2}u}=-\,2\iT{\mathrm{Hess}\,u\,:\,\frac{\nabla u\otimes\nabla u}u}+\iT{\frac{|\nabla u|^4}{u^2}}\,.
\]
Hence
\begin{multline*}
-\,\frac12\,\frac d{dt}\(\nrmT{\nabla u(t,\cdot)}2^2-\lambda\,\nrmT{u(t,\cdot)}2^2\)\\
=(2-p)\(\nrmT{\Delta u}2^2-\lambda\,\nrmT{\nabla u}2^2\)+(p-1)\,\left\|\mathrm{Hess}\,u-\tfrac{\nabla u\otimes\nabla u}u\right\|_{\mathrm L^2(\T)}^2\,.
\end{multline*}
We know from the Poincar\'e inequality that
\[
\nrmT{\Delta u}2^2\ge\nrmT{\nabla u}2^2\,,
\]
with optimal constant $1$, so we can conclude in the case $1\le p<2$ that $\nrmT{\nabla u(t,\cdot)}2^2-\lambda\,\nrmT{u(t,\cdot)}2^2$ is monotone nonincreasing if $0\le\lambda\le1$. As a consequence, we have the following result.
\begin{proposition}\label{Prop:BE} For any $p\in[1,2)$, we have
\[
\nrmT{\nabla u}2^2+\nrmT up^2\ge\nrmT u2^2\quad\forall\,u\in\mathrm H^1(\T)\,.
\]
\end{proposition}

\subsubsection{A tensorization result without magnetic potential}
A result better than Proposition~\ref{Prop:BE} follows from a tensorization argument that can be found in~\cite{MR2081075,DL2018}.
\begin{proposition}\label{Prop:tensorization} For any $p\in[1,2)$, we have
\be{Ineq:InterpZerotorus}
(2-p)\,\nrmT{\nabla u}2^2+\nrmT up^2\ge\nrmT u2^2\quad\forall\,u\in\mathrm H^1(\T)\,.
\ee
Moreover the factor $(2-p)$ is the optimal constant. \end{proposition}
\begin{proof} By taking on $\T$ a function depending only on $x\in\S^1$, it is clear that the constant in~\eqref{Ineq:InterpZerotorus} cannot be improved. The proof of~\eqref{Ineq:InterpZerotorus} can be done with the Bakry-Emery method applied to $\S^1$ and goes as follows.

Let us consider the flow given by
\[
\frac{\partial u}{\partial t}=u''+(p-1)\,\frac{|u'|^2}u
\]
and observe that $\frac d{dt}\,\nrmS{u(t,\cdot)}p^2=0$ on the one hand, and
\begin{multline*}
-\,\frac12\,\frac d{dt}\(\nrmS{u'(t,\cdot)}2^2-\lambda\,\nrmS{u(t,\cdot)}2^2\)\\
=\nrmS{u''}2^2+(p-1)\iSsigma{u''\,\frac{|u'|^2}u}-\lambda\,(2-p)\,\nrmS{u'}2^2\\
=\nrmS{u''}2^2+\frac13\,(p-1)\iSsigma{\frac{|u'|^4}{u^2}}-\lambda\,(2-p)\,\nrmS{u'}2^2
\end{multline*}
on the other hand. Hence
\[
-\,\frac12\,\frac d{dt}\(\nrmS{u'(t,\cdot)}2^2-\lambda\,\nrmS{u(t,\cdot)}2^2\)\le0
\]
if $\lambda\,(2-p)\le1$, because of the Poincar\'e inequality $\nrmS{u''}2^2\ge\nrmS{u'}2^2$. Up to a sign change of $\lambda$, this computation also holds if $p>2$ or if $p=-\,2$, as noticed in~\cite{doi:10.1063/1.5022121}, and it is straightforward to extend it to the limit case $p=2$ corresponding to the logarithmic Sobolev inequality.

According to~\cite[Proposition~3.1]{MR2081075} or~\cite[Theorem~2.1]{DL2018} and up to a straightforward adaptation to the periodic setting, the optimal constant for the inequality on $\T=\S^1\times\S^1$ is the same as for the inequality on $\S^1$, provided $1\le p<2$.
\end{proof}
As a consequence of Proposition~\ref{Prop:tensorization}, we have the inequality
\be{Interp:tensorization}
\nrmT{\nabla u}2^2+\mu\,\nrmT up^2\ge\Lambda_{0,p}(\mu)\,\nrmT u2^2\quad\forall\,u\in\mathrm H^1(\T)\,,
\ee
where $\mu\mapsto\Lambda_{0,p}(\mu)$ is a concave monotone increasing function on $(0,+\infty)$ such that $\Lambda_{0,p}(\mu)=\mu$ for any $\mu\in\big(0,1/(2-p)\big)$.

\subsubsection{A magnetic interpolation inequality in the flat torus}
Now let us consider the generalization of~\eqref{Interp:tensorization} to the case $a\neq0$.
\begin{lemma}\label{lem:InterpTorus} Assume that $p\in[1,2)$ and $a\in[0,1/2]$. With the notation~\eqref{nablaT2}, there exists a concave monotone increasing function $\mu\mapsto\Lambda_{a,p}(\mu)$ on $(0,+\infty)$ such that \hbox{$\lim_{\mu\to0_+}\Lambda_{a,p}(\mu)=a^2$} where $\Lambda_{a,p}(\mu)$ is the optimal constant in the inequality
\be{Ineq:Interp3}
\nrmT{\nA u}2^2+\mu\,\nrmT up^2\ge\Lambda_{a,p}(\mu)\,\nrmT u2^2\quad\forall\,u\in\HA(\T)\,.
\ee
Moreover, we have that
\[
\Lambda_{a,p}(\mu)\ge \mu+\big(1-\mu\,(2-p)\big)\,a^2\quad\mbox{for any}\quad\mu\le\frac1{2-p}\,.
\]
\end{lemma}
\begin{proof} For an arbitrary $t\in(0,1)$, we can write that
\begin{align*}
&\nrmT{\nA u}2^2+\mu\,\nrmT up^2\\
&\ge\,t\(\nrmT{\nA u}2^2-a^2\,\nrmT u2^2\)\\
&\kern12pt+(1-t)\(\nrmT{\nabla|u|}2^2+\frac\mu{1-t}\,\nrmT u2^p\)+t\,a^2\,\nrmT u2^2\\
&\ge\left[(1-t)\,\Lambda_{0,p}\(\frac\mu{1-t}\)+t\,a^2\right]\,\nrmT u2^2
\end{align*}
using the diamagnetic inequality $\nrmT{\nA u}2^2\ge\nrmT{\nabla|u|\,}2^2$. Inequality~\eqref{Interp:tensorization} applies with $\mu=1/(2-p)$ and $t=1-\mu\,(2-p)$.
\end{proof}

\subsubsection{A symmetry result in the subquadratic regime}
As an application of the results on magnetic rings of Theorem~\ref{prop:rigiditypless2}, we can prove a symmetry result for the optimal functions in~\eqref{Ineq:Interp3} in the case $p<2$. Let $\Lambda_{a,p}(\mu)$ be the optimal constant in~\eqref{Ineq:Interp3}.
\begin{theorem}\label{prop:InterpolationTorus} Assume that $a\in[0,1/2]$ and $p\in[1,2)$. Then
\[
\Lambda_{a,p}(\mu)=\lambda_{a,p}(\mu)\quad\mbox{if}\quad\mu\le\frac1{p-2}
\]
and any optimal function for~\eqref{Ineq:Interp} is then constant w.r.t.~$x$. Moreover, $\Lambda_{a,p}(\mu)=a^2+\mu$ if and only if $\mu\,(2-p)+4\,a^2\le1$ and equality in~\eqref{Ineq:Interp3} is then achieved only by the constants.\end{theorem}
\begin{proof} Let us use the notation $\fint f\,dx:=\frac1{2\pi}\int_{-\pi}^\pi f\,dx$ in order to denote a normalized integration with respect to the single variable $x$, where $y$ is considered as a parameter. For almost every $x\in\S^1$ we can apply~\eqref{Ineq:Interp} to the function $\psi(x,\cdot)$ and get
\begin{multline*}
\nrmT{\nA \psi}2^2+\mu\,\nrmT\psi p^2\\
\ge\nrmT{\partial_x\psi}2^2+\lambda_{a,p}(\mu)\,\nrmT\psi2^2+\mu\,\nrmT\psi p^2-\mu\fint\(\fint|\psi|^p\,dy\)^\frac2p\,dx
\end{multline*}
Let us define $u:=|\psi|$, $v(x):=\(\fint |u(x,y)|^p\,dy\)^{1/p}$ and observe that
\[
|v_x|=v^{1-p}\fint u^{p-1}\,u_x\,dy\le v^{1-p}\(\fint u^p\,dy\)^\frac{p-1}p\(\fint|u_x|^2\,dy\)^\frac12\(\fint 1\,dy\)^{\frac12-\frac{p-1}p}
\]
by H\"older's inequality, under the condition $p\le2$, that is,
\[
|v_x|^2\le\fint|u_x|^2\,dy\le\fint|\partial_x\psi|^2\,dy\,.
\]
We conclude that if $\mu\le1/(2-p)$,
\[
\iSsigma{|v_x|^2}+\mu\(\iSsigma{|v|^p}\)^{2/p}-\mu\iSsigma{|v|^2}+\lambda_{a,p}(\mu)\,\nrmT\psi2^2\ge\lambda_{a,p}(\mu)\,\nrmT\psi2^2
\]
using~\eqref{Ineq:InterpZero}. The equality is achieved by functions $v$ which are constant w.r.t.~$x$ and Theorem~\ref{prop:rigiditypless2} applies. \end{proof}

\subsection{Magnetic Hardy inequalities in dimensions \texorpdfstring{$2$ and $3$}{2 and 3}}\label{Sec:Consequences}

In this section, we draw some consequences of our results on magnetic rings of Section~\ref{Sec:magnetic rings $p<2$}. Here $d\sigma$ denotes the uniform probability measure on $\S^1$. The method relies on Keller-Lieb-Thirring dual estimates.

\subsubsection{Keller-Lieb-Thirring inequalities on the circle}
As in~\cite{MR3784917}, by duality we obtain a spectral estimate.
\begin{proposition}\label{Prop:KLT} Assume that $a\in[0,1/2]$ and $p\in[1,2)$. If $\phi$ is a nonnegative potential such that $\phi^{-1}\in\mathrm L^q(\S^1)$, then the lowest eigenvalue $\lambda_1$ of $-\,(\partial_y-\,i\,a)^2+\phi$ is bounded from below according to
\[
\lambda_1\ge\lambda_{a,p}\(\nrmS{\phi^{-1}}q^{-1}\)
\]
and equality is achieved by a constant potential $\phi$ if $\nrmS{\phi^{-1}}q^{-1}\,(2-p)+4\,a^2\le1$.\end{proposition}
\begin{proof} Using H\"older's inequality with exponents $2/(2-p)$ and $2/p$, we get that
\[
\nrmS\psi p^2=\(\iSsigma{\phi^{-\frac p2}\,\(\phi\,|\psi|^2\)^\frac p2}\)^{2/p}\le\nrmS{\phi^{-1}}q\,\iSsigma{\phi\,|\psi|^2}
\]
with $q=p/(2-p)$, and with $\mu=\nrmS{\phi^{-1}}q^{-1}$,
\begin{multline}\label{Ineq:KLT2pectral}
\iSsigma{|\psi'-i\,a\,\psi|^2}+\iSsigma{\phi\,|\psi|^2}\ge\iSsigma{|\psi'-i\,a\,\psi|^2}+\mu\,\nrmS\psi p^2\\
\ge\lambda_{a,p}(\mu)\iSsigma{|\psi|^2}\,.
\end{multline}
If $\phi$ is constant, then there is equality in H\"older's inequality.\end{proof}

The spectral estimate~\eqref{Ineq:KLT2pectral} is of a different nature than~\eqref{Ineq:KLT1spectral} because the potential energy and the magnetic kinetic energy have the same sign. By considering the threshold case $\mu\,(2-p)+4\,a^2=1$, we obtain an interesting estimate.
\begin{corollary}\label{Cor:KLT2} Let $a\in[0,1/2]$, $p\in(1,2)$ and $q=p/(2-p)$. If $\phi$ is a nonnegative potential such that $\phi^{-1}\in\mathrm L^q(\S^1)$, then
\begin{multline*}
\iSsigma{|\psi'-i\,a\,\psi|^2}+\frac{1-4\,a^2}{2-p}\,\nrmS{\phi^{-1}}q\iSsigma{\phi\,|\psi|^2}\\
\ge\(\frac{1-4\,a^2}{2-p}+a^2\)\nrmS\psi2^2\quad\forall\,\psi\in\mathrm H^1(\S^1)\,.
\end{multline*}
\end{corollary}

\subsubsection{Magnetic Hardy-type inequalities in dimensions two and three}
Let us denote by $\theta\in[-\pi,\pi)$ the angular coordinate associated with $x\in\R^2$. As in~\cite{doi:10.1063/1.5022121}, we can deduce a Hardy-type inequality for Aharonov-Bohm magnetic potentials in dimension $d=2$.
\begin{corollary}\label{Cor:Hardy2} Let $\mathbf A$ as in~\eqref{AB2}, $a\in[0,1/2]$, $p\in(1,2)$ and $q=p/(2-p)$. If $\phi$ is a nonnegative potential such that $\phi^{-1}\in\mathrm L^q(\S^1)$ with $\nrmS{\phi^{-1}}q=1$, then for any complex valued function $\psi\in\mathrm H^1(\R^2)$ we have
\[
\ir2{|\nA\psi|^2}+\frac{1-4\,a^2}{2-p}\ir2{\frac{\phi(\theta)}{|x|^2}\,|\psi(x)|^2}\ge\(\frac{1-4\,a^2}{2-p}+a^2\)\ir2{\frac{|\psi|^2}{|x|^2}}\,.
\]
\end{corollary}
Let us consider cylindrical coordinates $(\rho,\theta,z)\in\R^+\times[0,2\pi)\times\R$ such that $|x|^2=\rho^2+z^2$. In this system of coordinates the magnetic kinetic energy is
\[
\ird{|\nA\psi|^2}=\ird{\(\big|\tfrac{\partial\psi}{\partial\rho}\big|^2+\tfrac1{\rho^2}\big|\tfrac{\partial\psi}{\partial\theta}-\,i\,a\,\psi\big|^2+\big|\tfrac{\partial\psi}{\partial z}\big|^2\)}
\]
where $d\mu:=\rho\,d\rho\,d\theta\,dz$. The following result was proved in~\cite[Section~2.2]{MR4134390}.
\begin{lemma}\label{lem:Hardy} For any $\psi\in\mathrm H^1(\R^3)$, we have
\[
\irdmu{\(\big|\tfrac{\partial\psi}{\partial\rho}\big|^2+\big|\tfrac{\partial\psi}{\partial z}\big|^2\)}\ge\frac14\irdmu{\frac{|\psi|^2}{\rho^2+z^2}}\quad\forall\,\psi\in\mathrm H^1(\R^3)\,.
\]
\end{lemma}
\begin{proof} We give an elementary proof. Assume that $\psi$ is smooth and has compact support. The inequality follows from the expansion of the square
\[
\irdmu{\(\big|\tfrac{\partial\psi}{\partial\rho}+\tfrac{\rho\,\psi}{2\,(\rho^2+z^2)}\big|^2+\big|\tfrac{\partial\psi}{\partial z}+\tfrac{z\,\psi}{2\,(\rho^2+z^2)}\big|^2\)}\ge0\quad\forall\,\psi\in\mathrm H^1(\R^3)
\]
and of an integration by parts of the cross terms.\end{proof}

Lemma~\ref{lem:Hardy} is an improved version of the standard Hardy inequality in the sense that the left-hand side of the inequality does not involve the angular part of the kinetic energy. A consequence of Corollary~\ref{Cor:KLT2} and Lemma~\ref{lem:Hardy} is a Hardy-like estimate in dimension $d=3$. For the angular part we argue as in Corollary~\ref{Cor:Hardy2}. Details of the proof are left to the reader.
\begin{theorem}\label{Thm:ImprovedHardy2} Let $\mathbf A$ as in~\eqref{AB3}, $a\in[0,1/2]$, $p\in(1,2)$ and $q=p/(2-p)$. If $\phi$ is a potential such that $\phi^{-1}\in\mathrm L^q(\S^1)$ with $\nrmS{\phi^{-1}}q=1$, then for any complex valued function $\psi\in\mathrm H^1(\R^3)$ we have
\begin{multline*}
\ird{|\nA\psi|^2}+\frac{1-4\,a^2}{2-p}\ird{\frac{\phi(\theta)}{\rho^2}\,|\psi(x)|^2}\\
\ge\frac14\ird{\frac{|\psi|^2}{|x|^2}}+\(\frac{1-4\,a^2}{2-p}+a^2\)\ird{\frac{|\psi|^2}{\rho^2}}\,.
\end{multline*}
\end{theorem}
A simple case is $\phi\equiv1$, for which we obtain that
\[
\ird{|\nA\psi|^2}\ge\frac14\ird{\frac{|\psi|^2}{|x|^2}}+a^2\ird{\frac{|\psi|^2}{|\rho|^2}}\quad\forall\,\psi\in\HA(\R^3)\,.
\]

\section{Aharonov-Bohm magnetic interpolation inequalities in \texorpdfstring{$\R^2$}{R2}}\label{Sec:Interpolation $p>2$}

Magnetic interpolation inequalities on $\R^2$ are considered without weights in Section~\ref{Sec:UnweightedR2}. Weights are then introduced as in~\cite{HS-magnetic-sharp} in order to prove the new magnetic Caffarelli-Kohn-Nirenberg inequality of Corollary~\ref{Cor:magneticCaffarelli-Kohn-Nirenberg} in Section~\ref{Sec:Hardy-Sobolev} and a magnetic Hardy inequality on $\R^2$ in Theorem~\ref{Cor:Prop:Hardy2} (Section~\ref{Subsec:4.3}).

\subsection{Magnetic interpolation inequalities without weights}\label{Sec:UnweightedR2}

Let us consider on $\R^2$ the Aharonov-Bohm magnetic potential $\mathbf A$ given by~\eqref{AB2}. Using the \emph{diamagnetic inequality}
\[
|\nA\psi|^2\ge\big|\nabla|\psi|\big|^2\quad\mbox{a.e. in }\R^2
\]
and, for any $p\in(2,\infty)$ and $\lambda>0$, the Gagliardo-Nirenberg inequality
\be{GN2}
\nrmRdeux{\nabla\psi}2^2+\lambda\,\nrmRdeux\psi p^2\ge\mathsf C_p\,\lambda^\frac p2\,\nrmRdeux\psi2^2\quad\forall\,\psi\in\mathrm H^1(\R^2)\cap\mathrm L^p(\R^2)
\ee
with optimal constant $\mathsf C_p$, we deduce that
\be{Ineq:InterpL2}
\nrmRdeux{\nA\psi}2^2+\lambda\,\nrmRdeux\psi2^2\ge\mu_{a,p}(\lambda)\,\nrmRdeux\psi p^2\quad\forall\,\psi\in\mathrm H_a^1(\R^2)\,.
\ee
See~\cite[Section~3]{MR3784917} for details. Here $\mu_{a,p}(\lambda)$ is the optimal constant in~\eqref{Ineq:InterpL2} for any given $a$, $p$ and $\lambda$ and, as a function of $\lambda$, $\mu_{a,p}(\lambda)$ is monotone increasing and concave. Notice that right-hand sides in~\eqref{GN2} and~\eqref{Ineq:InterpL2} involve norms with respect to Lebesgue's measure. It turns out that $\mu_{a,p}(\lambda)$ is equal to the best constant of the non-magnetic problem.
\begin{proposition}\label{Prop:nomin} Let $a\in\R$ and $p\in(2,\infty)$. The optimal constant in~\eqref{Ineq:InterpL2} is
\[
\mu_{a,p}(\lambda)=\mathsf C_p\,\lambda^\frac p2\quad\forall\,\lambda>0
\]
and equality is not achieved on $\mathrm H^1(\R^2)\cap\mathrm L^p(\R^2)$ if $a\in\R\setminus\Z$.\end{proposition}
\begin{proof} By construction we know that $\mu_{a,p}(\lambda)\ge\mathsf C_p\,\lambda^{p/2}$. By taking an optimal function $\psi$ for~\eqref{GN2} and considering $\psi_n(x)=\psi(x+n\,\mathsf e)$ with $n\in\N$ and $\mathsf e\in\S^1$, we see that there is equality.

Let us prove by contradiction that equality is not achieved. If $\psi\in\mathrm H^1(\R^2)\cap\mathrm L^p(\R^2)$ is optimal, let $\phi=e^{\,i\,a\,\theta}\psi$. Since
\[
\nrmRdeux{\nA\psi}2^2=\ir2{|\partial_r\psi|^2+\frac{|\partial_\theta\phi|^2}{|x|^2}}
\]
and equality in~\eqref{GN2} is achieved by functions with a constant phase only, this means that $\partial_\theta\phi=0$ a.e., a contradiction with the periodicity of $\psi$ with respect to $\theta\in[0,2\pi)$ if $a\not\in\Z$.\end{proof}

Proposition~\ref{Prop:nomin} means that the Aharonov-Bohm magnetic potential plays no role in non-weighted interpolation inequalities. This is why it is natural to introduce weighted norms with adapted scaling properties.

\subsection{Magnetic Caffarelli-Kohn-Nirenberg inequalities in \texorpdfstring{$\R^2$}{R2}}\label{Sec:Hardy-Sobolev}

The \emph{Caffarelli-Kohn-Nirenberg inequality}
\be{CKN}
\ir2{\frac{|\nabla v|^2}{|x|^{2\mathsf a}}}\ge\mathsf C_{\mathsf a}\(\ir2{\frac{|v|^p}{|x|^{\mathsf b\,p}}}\)^{2/p}\quad\forall\,v\in\mathcal D(\R^2)
\ee
has been established in~\cite{Caffarelli-Kohn-Nirenberg-84} and, earlier, in~\cite{Ilyin}. The exponent $\mathsf b=\mathsf a+2/p$ is determined by the scaling invariance and as $p$ varies in $(2,\infty)$, the parameters $\mathsf a$ and $\mathsf b$ are such that $\mathsf a<\mathsf b\le\mathsf a+1$ and $\mathsf a<0$. The case $\mathsf a>0$ can be considered in an appropriate functional space after a Kelvin-type transformation: see~\cite{Catrina-Wang-01,DELT09}, but we will not consider this case here. As noticed for instance in~\cite{DELT09}, by considering $v(x)=|x|^{\mathsf a}\,u(x)$, Ineq.~\eqref{CKN} is equivalent to the \emph{Hardy-Sobolev inequality}
\be{CKN2}
\ir2{|\nabla u|^2}+\mathsf a^2\ir2{\frac{|u|^2}{|x|^2}}\ge\mathsf C_{\mathsf a}\(\ir2{\frac{|u|^p}{|x|^2}}\)^{2/p}\quad\forall\,u\in\mathcal D(\R^2)\,.
\ee
The optimal functions for~\eqref{CKN} are radially symmetric if and only if
\[
\mathsf b\ge\mathsf b_{\rm FS}(\mathsf a):=\mathsf a-\frac{\mathsf a}{\sqrt{1+\mathsf a^2}}
\]
according to~\cite{Felli-Schneider-03,MR3570296}. We refer to~\cite{HS-magnetic-sharp} for more details and for the proof of the following magnetic Hardy-Sobolev inequality.
\begin{theorem}[\cite{HS-magnetic-sharp}]\label{Thm:MgnHS} Let $a\in[0,1/2]$, $\mathbf A$ as in~\eqref{AB2} and $p>2$. For any $\lambda>-\,a^2$, there is an optimal function $\lambda\mapsto\mu(\lambda)$ which is monotone increasing and concave such that
\be{Ineq:MgnHS}
\ir2{|\nA\psi|^2}+\lambda\ir2{\frac{|\psi|^2}{|x|^2}}\ge\mu(\lambda)\(\ir2{\frac{|\psi|^p}{|x|^2}}\)^{2/p}\quad\forall\,\psi\in\HA(\R^2)\,.
\ee
If $a\in [0, 1/2)$, the optimal function in~\eqref{Ineq:MgnHS} is
\[
\psi(x)=\(|x|^\alpha+|x|^{-\alpha}\)^{-\frac2{p-2}}\quad\forall\,x\in\R^2\,,\quad\mbox{with}\quad\alpha=\frac{p-2}2\,\sqrt{\lambda+a^2}\,,
\]
up to a scaling and a multiplication by a constant, if
\[
\lambda\le\lambda_\star:=4\,\frac{1-4\,a^2}{p^2-4}-\,a^2\,.
\]
Conversely, if $a\in [0, 1/2]$ and $\lambda>\lambda_\bullet$ with
\[\label{Eqn:SB}
\lambda_\bullet:=\frac{8\(\sqrt{p^4-a^2\,(p-2)^2\,(p+2)\,(3\,p-2)}+2\)-4\,p\,(p+4)}{(p-2)^3\,(p+2)}-\,a^2\,,
\]
there is symmetry breaking, \emph{i.e.}, the optimal functions are not radially symmetric.
\end{theorem}
An explicit computation shows that $\lambda_\star<\lambda_\bullet$ for any $a\in(0,1/2)$, and so there is a zone where we do not know whether the optimal functions in~\eqref{Ineq:MgnHS} are symmetric or not. Nevertheless, as shown in \cite{HS-magnetic-sharp}, the values of $\lambda_\star$ and $ \lambda_\bullet$ are numerically very close to each other. If $\lambda\le\lambda_\star$, the expression of $\mu(\lambda)$ is explicit and given by
\[\label{mu-lambda}
\mu(\lambda)=\frac p2\,(2\,\pi)^{1-\frac2p}\(\lambda+a^2\)^{1+\frac2p}\(\frac{2\,\sqrt\pi\,\Gamma\big(\frac p{p-2}\big)}{(p-2)\,\Gamma\big(\frac p{p-2}+\frac12\big)}\)^{1-\frac2p}\,.
\]
See~\cite[Appendix]{HS-magnetic-sharp} for the details of the computation of the constant.

Inspired by the equivalence of~\eqref{CKN} and~\eqref{CKN2}, we prove that the magnetic Hardy-Sobolev inequality~\eqref{Ineq:MgnHS} is equivalent to an interpolation inequality of Caffarelli-Kohn-Nirenberg type in the presence of the Aharonov-Bohm magnetic field.
\begin{corollary}[Magnetic Caffarelli-Kohn-Nirenberg inequality]\label{Cor:magneticCaffarelli-Kohn-Nirenberg} Let $p\in(2,+\infty)$ and $\mathbf A$ as in~\eqref{AB2} for some $a\in[0,1/2]$ and $\mathsf a\le0$.
With $\mu$ as in Theorem~\ref{Thm:MgnHS}, for any $\gamma<\mathsf a^2+a^2$, we have that
\[
\ir2{\frac{|\nA\phi|^2}{|x|^{2\mathsf a}}}\ge\gamma\,\ir2{\frac{|\phi|^2}{|x|^{2\mathsf a+2}}}+\mu(\mathsf a^2-\gamma)\,\(\ir2{\frac{|\phi|^p}{|x|^{\mathsf a\,p+2}}}\)^{2/p}\quad\forall\,\phi\in\mathcal D(\R^2;\C)
\]
and $\mu(\lambda)$ with $\lambda=\mathsf a^2-\gamma$ is the optimal constant.\end{corollary}
The cases of symmetry and symmetry breaking in Theorem~\ref{Thm:MgnHS} have their exact counterpart in Corollary~\ref{Cor:magneticCaffarelli-Kohn-Nirenberg}. Details are left to the reader.
\begin{proof} Let us consider the function $\phi(x)=|x|^{\mathsf a}\,\psi(x)$ and observe that
\[
\ir2{\frac{|\nA\phi|^2}{|x|^{2\mathsf a}}}=\ir2{|\nA\psi|^2}+\mathsf a^2\ir2{\frac{|\psi|^2}{|x|^2}}
\]
and conclude by applying~\eqref{Ineq:MgnHS} to $\psi$ with $\lambda=\mathsf a^2-\gamma$.\end{proof}

\subsection{A magnetic Hardy inequality in \texorpdfstring{$\R^2$}{R2}}\label{Subsec:4.3}
Another consequence of Theorem~\ref{Thm:MgnHS} is the following \emph{magnetic Keller-Lieb-Thirring inequality}, which can be found in~\cite[Theorem~1]{HS-magnetic-sharp}. Let $q=p/(p-2)$. The ground state energy $\lambda_1$ of the magnetic Schr\"odinger operator $-\LapA-\phi$ on $\R^2$ is such that
\be{KLTR2}
\lambda_1(-\LapA-\phi)\ge-\,\lambda\(\mu\)\quad\mbox{where}\quad\mu=\(\ir2{|\phi|^q\,|x|^{2\,(q-1)}}\)^{1/q}
\ee
and $\mu\mapsto\lambda(\mu)$ is a convex monotone increasing function on~$\R^+$ such that $\lim_{\mu\to0_+}\lambda(\mu)=-\,a^2$, defined as the inverse of $\lambda\mapsto\mu(\lambda)$ of Theorem~\ref{Thm:MgnHS}. Again $\lambda\(\mu\)$ is optimal in~\eqref{KLTR2} and the cases of symmetry and symmetry breaking are in correspondence with the ones of Theorem~\ref{Thm:MgnHS}.

Alternatively, let us consider a function $\phi$ on $\R^2$. We can estimate an associated magnetic Schr\"odinger energy from below by
\[
\ir2{\(|\nA\psi|^2-\tau\,\frac{\phi}{|x|^2}\,|\psi|^2\)}\ge\ir2{|\nA\psi|^2}-\,\tau\(\ir2{\frac{|\phi|^q}{|x|^2}}\)^\frac1q\(\ir2{\frac{|\psi|^p}{|x|^2}}\)^\frac2p
\]
by H\"older's inequality, with $q=p/(p-2)$, for an arbitrary parameter $\tau>0$. For an appropriate choice of $\tau$, we obtain the following result.
\begin{theorem}[A magnetic Hardy inequality]\label{Cor:Prop:Hardy2} Let $q\in(1,2)$ and $\mathbf A$ as in~\eqref{AB2} for some $a\in[0,1/2]$. Then for any function $\phi\in\mathrm L^q\(\R^2\,|x|^{-2}\,dx\)$, we have
\[
\ir2{|\nA\psi|^2}\ge\mu(0)\(\ir2{\frac{|\phi|^q}{|x|^2}}\)^{-\frac1q}\ir2{\frac{\phi}{|x|^2}\,|\psi|^2}\quad\forall\,\psi\in\HA(\R^2)\,,
\]
where $\mu(\cdot)$ is the best constant in \eqref{Ineq:MgnHS}. Finally, when $\;a^2\le 4/(12+p^2)$, we know the value of $\mu(0)$ explicitly:
\[
\mu(0)=\frac p2\,(2\,\pi)^{1-\frac2p}\,a^{2+\frac4p}\(\frac{2\,\sqrt\pi\,\Gamma\big(\frac p{p-2}\big)}{(p-2)\,\Gamma\big(\frac p{p-2}+\frac12\big)}\)^{1-\frac2p}\,.
\]\end{theorem}

\section{Aharonov-Bohm magnetic Hardy inequalities in \texorpdfstring{$\R^3$}{R3}}\label{Sec:Hardy-Aharonov-Bohm}

In this section we address the issue of improved magnetic Hardy inequalities with the Aharonov-Bohm magnetic potential in dimension $d=3$ as defined by~\eqref{AB3}. Our results improve upon~\cite[Section~V.B]{MR3202866}, including the case of a constant magnetic field.

\subsection{An improved Hardy inequality with radial symmetry}\label{Sec:Hardy-radial}
In~\cite[Section~V.B]{MR3202866}, it is proved that \emph{for all $a>0$, there is a constant $\mathcal C(a)$ such that $\mathcal C(a)=a^2$ if $a\in[0,1/2]$ and}
\be{Ekholm-Portmann}
\ird{|\nA\psi|^2}\ge\big(\tfrac14+\mathcal C(a)\big)\ird{\frac{|\psi|^2}{|x|^2}}\quad\forall\,\psi\in\HA(\R^3)\,.
\ee
If we allow for an angular dependence, we have the following result.
\begin{theorem}\label{Thm:Aharonov-Bohm} Let $\mathbf A$ as in~\eqref{AB3}, $a\in[0,1/2]$ and $q\in(1,+\infty)$. Then, for all $\pot\in\mathrm L^q(\S^2)$,
\[
\ird{|\nA\psi|^2}\ge\ird{\(\frac14+\frac{\mu_{a,p}(0)}{\nrmSdeux \pot q}\,\pot(\omega)\)\,\frac{|\psi|^2}{|x|^2}}\quad\forall\,\psi\in\HA(\R^3)\,.
\]
Here $\omega=x/|x|$ and $\mu_{a,p}$ is defined as in Proposition~\ref{prop:interpolation}.
\end{theorem}
In the case $a\in[0,1/2]$, according to Proposition~\ref{prop:interpolation}, we find in the limit case as $p\to2_+$ that $\mu_{a,2}(0)\ge\Lambda_a=a\,(a+1)$ and improve the estimate~\eqref{Ekholm-Portmann} to $\mathcal C(a)=a\,(a+1)$ if $\pot\equiv1$.
\begin{proof} Let us use spherical coordinates $(r,\omega)\in[0,+\infty)\times\S^2$. The result follows from an expansion of the square and an integration by parts in
\[
0\le\int_0^{+\infty}\left|\partial_r\psi+\tfrac1{2\,r}\,\psi\right|^2\,r^2\,dr=\int_0^{+\infty}|\partial_r\psi|^2\,r^2\,dr-\tfrac14\int_0^{+\infty}|\psi|^2\,dr
\]
for the radial part of the Dirichlet integral, and from Corollary~\ref{Cor:KLT22} for the angular part.\end{proof}

\subsection{An improved Hardy inequality with cylindrical symmetry}\label{Sec:Hardy-cylindrical}

The improved Hardy inequality (without angular kinetic energy) of Lemma~\ref{lem:Hardy} and~\eqref{Ineq:KLT1} can be combined into the following improved Hardy inequality in presence of a magnetic potential.
\begin{theorem}\label{Thm:ImprovedHardy} Let $\mathbf A$ as in~\eqref{AB3}, $a\in[0,1/2]$, $p>2$, $q=p/(p-2)$ and $\phi\in\mathrm L^q(\S^1)$. For any $\psi\in\HA(\R^3)$, we have
\[
\ird{|\nA\psi|^2}\ge\frac14\ird{\frac{|\psi|^2}{|x|^2}}+\frac{\mu_{a,p}(0)}{\nrmS\phi q}\irdmu{\frac{\phi(\theta)}{\rho^2}\,|\psi(\rho,\theta,z)|^2}\,.
\]
\end{theorem}
Notice that the inequality is a strict improvement upon the Hardy inequality without a magnetic potential combined with the diamagnetic inequality. A simple case which is particularly illuminating is $\phi\equiv1$ with $a^2\le1/(p+2)$ so that \hbox{$\mu_{a,p}(0)=a^2$} according to Proposition~\ref{prop:rigidity}, in which case we obtain that
\[
\ird{|\nA\psi|^2}\ge\frac14\ird{\frac{|\psi|^2}{|x|^2}}+a^2\ird{\frac{|\psi|^2}{|\rho|^2}}\quad\forall\,\psi\in\HA(\R^3)\,.
\]

\section*{Acknowledgments}\begin{spacing}{0.9}
{\small This research has been partially supported by the project \emph{EFI}, contract~ANR-17-CE40-0030 (D.B., J.D.) of the French National Research Agency (ANR), by the PDR (FNRS) grant T.1110.14F and the ERC AdG 2013 339958 ``Complex Patterns for Strongly Interacting Dynamical Systems - COMPAT'' grant (D.B.), by the RSF grant 18-11-00032 (A.L.) and by the NSF grant DMS-1600560 (M.L.). The authors thank an anonymous referee for constructive remarks which significantly improved the presentation of the results.\\
\scriptsize\copyright\,2020 by the authors. This paper may be reproduced, in its entirety, for non-commercial purposes.}\end{spacing}


\bigskip\begin{flushright}\today\end{flushright}

\end{document}